\documentclass{amsart}

\usepackage{amssymb,amsmath,enumerate}

\usepackage[all]{xy}
\SelectTips{cm}{}

\newtheorem{theorem}[subsection]{Theorem}

\newtheorem*{TNT}{Tensor Nilpotence Theorem}
\newtheorem*{NNT}{Morphic Intersection Theorem}
\newtheorem{proposition}[subsection]{Proposition}
\newtheorem{lemma}[subsection]{Lemma}
\newtheorem{corollary}[subsection]{Corollary}

\theoremstyle{definition}

\theoremstyle{remark}
\newtheorem{remark}[subsection]{Remark}

\numberwithin{equation}{subsection}

\newcommand{\ann}{\operatorname{ann}}

\newcommand{\chcat}[1]{{\mathsf{C}(#1)}}
\newcommand{\dcat}[1]{{\mathsf{D}(#1)}}
\newcommand{\dd}{\partial}

\newcommand{\filt}[2]{{#2}(#1)}

\newcommand{\height}{\operatorname{height}}
\newcommand{\hh}[1]{\operatorname{H}(#1)}
\newcommand{\HH}[2]{\operatorname{H}_{#1}(#2)}
\newcommand{\Hom}[3]{\operatorname{Hom}_{#1}(#2,#3)}
\newcommand{\LC}[3]{\operatorname{H}^{#1}_{#2}(#3)}
\newcommand{\lch}{\operatorname{R}\!\Gamma}

\newcommand{\level}[3][{}]{\operatorname{level}_{#1}^{#2}{#3}}
\newcommand{\lotimes}[1]{\otimes^{\operatorname{L}}_{#1}}
\newcommand{\ov}{\overline}

\newcommand{\Rhom}[3]{\operatorname{RHom}_{#1}({#2},{#3})}
\newcommand{\cspan}[1]{\operatorname{span}{#1}}

\newcommand{\wt}{\widetilde}

\newcommand{\bbN}{{\mathbb N}}
\newcommand{\bbZ}{{\mathbb Z}}

\newcommand{\lra}{\longrightarrow}

\newcommand{\xra}{\xrightarrow}

\newcommand{\ges}{\geqslant}
\newcommand{\les}{\leqslant}

\newcommand{\nat}[1]{{#1}^{\natural}}

\newcommand{\bsx}{{\boldsymbol x}}
\newcommand{\bsy}{{\boldsymbol y}}

\newcommand{\Tor}[4]{\operatorname{Tor}_{#1}^{#2}(#3,#4)}
\newcommand\fm{{\mathfrak m}}

\newcommand\fp{{\mathfrak p}}

\newcommand\Image{\operatorname{Im}}

\newcommand\Ker{\operatorname{Ker}}

\newcommand{\rank}{\operatorname{rank}}
\newcommand{\shift}{{\sf\Sigma}}
\newcommand\Spec{\operatorname{Spec}}

\begin{document}

\title[Tensor nilpotent morphisms]{Big Cohen-Macaulay modules, \\ morphisms of perfect complexes, and \\ intersection theorems in local algebra}

\author[L.~L.~Avramov]{Luchezar L.~Avramov}
\address{Department of Mathematics,
University of Nebraska, Lincoln, NE 68588, U.S.A.}
\email{avramov@math.unl.edu}

\author[S.~B.~Iyengar]{Srikanth B.~Iyengar}
\address{Department of Mathematics,
University of Utah, Salt Lake City, UT 84112, U.S.A.}
\email{iyengar@math.utah.edu}

\author[A.~Neeman]{Amnon Neeman}
\address{Centre for Mathematics and its Applications, Mathematical Sciences Institute
Australian National University, Canberra, ACT 0200, Australia.}
\email{Amnon.Neeman@anu.edu.au}

\thanks{Partly supported by NSF grants DMS-1103176 (LLA) and DMS-1700985 (SBI)}

\date{\today}


\begin{abstract} 
There is a well known link from the first topic in the title to the third one.  In this paper we thread that link through the second topic.
The central result is a criterion for the tensor nilpotence of morphisms of perfect complexes over commutative noetherian rings, in 
terms of a numerical invariant of the complexes known as their level.  Applications to local rings include a strengthening of the 
Improved New Intersection Theorem, short direct proofs of several results equivalent to it, and lower bounds on the ranks of the 
modules in every finite free complex that admits a structure of differential graded module over the Koszul complex on some 
system of parameters.
\end{abstract}

\maketitle

\section{Introduction}

A big Cohen-Macaulay module over a commutative noetherian local ring $R$ is a (not necessarily finitely generated) $R$-module $C$ such that some system of parameters of $R$ forms a $C$-regular sequence.  In \cite{Hoc1} Hochster showed that the existence of such modules implies several fundamental homological properties of finitely generated $R$-modules.  In \cite{Hoc2}, published in \cite{Hoc3}, he proved that big Cohen-Macaulay modules exist for algebras over fields, and conjectured their existence in the case of mixed characteristic.  This was recently proved by Y.~Andr\'e in \cite{An}; as a major consequence many ``Homological Conjectures'' in local algebra are now theorems.

A perfect $R$-complex is a bounded complex of finite projective $R$-modules.  Its level with respect to $R$, introduced in \cite{ABIM} and defined in \ref{ssec:levels}, measures the minimal number of mapping cones needed to assemble a quasi-isomorphic complex from 
bounded complexes of finite projective modules with differentials equal to zero.  

The main result of this paper, which appears as Theorem~\ref{thm:fiberwisezero}, is the following

  \begin{TNT}
Let $f\colon G\to F$ be a morphism of perfect complexes over a commutative noetherian ring $R$.

If $f$ factors through a complex whose homology is $I$-torsion for some ideal $I$ of $R$ with  
$\height I\ge \level{R}{\Hom RGF}$, then the induced morphism 
  \[
\otimes^n_R\,f\colon \otimes^n_R G \to \otimes^n_R\,F  
  \]
is homotopic to zero for some non-negative integer $n$.
  \end{TNT}

\noindent 
Big Cohen-Macaulay modules play an essential, if discreet role in the proof, as a tool for constructing special 
morphisms in the derived category of $R$; see Proposition~\ref{prp:BCM}.

In applications to commutative algebra it is convenient to use another property of morphisms of perfect complexes:  $f$ is \emph{fiberwise zero} if $\hh{k(\fp)\otimes_{R}f}=0$ holds for every $\fp$ in $\Spec R$.  Hopkins~\cite{Hop} and Neeman~\cite{Ne} have shown that this is equivalent to tensor nilpotence; this is a key tool for the classification of the thick subcategories of perfect $R$-complexes. 

It is easy to see that the level of a complex does not exceed its span, defined in~\ref{ssec:complexes}.  Due to these remarks, the Tensor Nilpotence Theorem is equivalent to the

  \begin{NNT}
If $f$ is not fiberwise zero and factors through a complex with $I$-torsion homology for an ideal $I$ of $R$, then
there are inequalities:
\[
\cspan F + \cspan G -1 \ge  \level{R}{\Hom RGF} \ge  \height I + 1\,.
\]
  \end{NNT}

In Section~\ref{sec:applications} we use this result to prove directly, and sometimes to generalize and sharpen, several basic theorems in commutative algebra.  These include the Improved New Intersection Theorem, the  Monomial Theorem and several versions of the Canonical Element Theorem.  All of them are equivalent, but we do not know if they imply the Morphic Intersection Theorem; a potentially significant obstruction to that  is that the difference $\cspan F-\level RF$ can be arbitrarily big.

Another application, in Section~\ref{sec:rank}, yields lower bounds for ranks of certain finite free complexes, 
related to a conjecture of Buchsbaum and Eisenbud, and Horrocks.

In \cite{AIN} a version of the Morphic Intersection Theorem for certain tensor triangulated categories is proved.  This has implications for morphisms of perfect complexes of sheaves and, more generally, of perfect differential sheaves, over schemes.

\section{Perfect complexes}
\label{sec:invariants}
Throughout this paper $R$ will be a commutative noetherian ring. 

This section is a recap on the various notions and construction, mainly concerning perfect complexes, needed in this work. Pertinent references include \cite{ABIM, Ro0}.

\subsection{Complexes}
\label{ssec:complexes}
In this work, an $R$-complex (a shorthand for `a complex of $R$-modules') is a sequence of 
homomorphisms of $R$-modules
\[
X:= \quad \cdots \lra X_{n}\xra{\ \dd^{X}_{n}\ } X_{n-1} \xra{\dd^{X}_{n-1}} X_{n-2}\lra \cdots 
\]
such that $\dd^{X} \dd^{X}=0$. We write $\nat X$ for the graded $R$-module underlying $X$. The $i$th suspension of $X$ is the 
$R$-complex $\shift^i X$ with $(\shift^i X)_{n}=X_{n-i}$ and $\dd^{\shift^i X}_n=(-1)^i\dd^{X}_{n-i}$ for each~$n$. The \emph{span} of $X$ is the number 
\[
\cspan X:= \sup\{i\mid X_{i}\ne 0\} - \inf\{i\mid X_{i}\ne 0\} + 1
\]
Thus $\cspan X=-\infty$ if and only if $X=0$, and $\cspan X=\infty$ if and only if $X_{i}\ne 0$ for infinitely many 
$i\ge 0$.  The span of $X$ is finite if and only if $\cspan X$ is an natural number. When $\cspan X$ is finite we 
say that $X$ is \emph{bounded}.

Complexes of $R$-modules are objects of two distinct categories. 

In the category of complexes $\chcat R$ a morphism $f\colon Y\to X$ of $R$-complexes is a family 
$(f_i\colon Y_i\to X_i)_{i\in\bbZ}$ or $R$-linear maps satisfying $\dd^X_if_i=f_{i-1}\dd^Y_i$.  It is a 
\emph{quasi-isomorphism} if $\hh f$, the induced map in homology, is bijective. Complexes that 
can be linked by a string of quasi-isomorphisms are said to be \emph{quasi-isomorphic}. 

The derived category $\dcat R$ is obtained from $\chcat R$ by inverting all quasi-isomor\-phisms.
For constructions of the localization functor $\chcat R\to\dcat R$ and of the derived functors 
$?\lotimes R?$ and $\Rhom R{\text{>}}?$, see e.g.~\cite{GM, We, Li}.  When $P$ is a complex of 
projectives with $P_i=0$ for $i\ll0$, the functors $P\lotimes R?$ and $\Rhom RP?$ are represented 
by $P\otimes_{R}?$ and $\Hom RP?$, respectively.  In particular, the localization functor induces 
for each $n$ a natural isomorphism of abelian groups 
  \begin{equation}
    \label{eq:complexes1}
\HH {-n}{\Rhom RPX}\cong\Hom{\dcat R}P{\shift^{n}X}\,.
  \end{equation}

\subsection{Perfect complexes}
In $\chcat R$, a \emph{perfect} $R$-complex is a bounded complex of finitely generated projective $R$-modules.
When $P$ is perfect, the $R$-complex $P^{*}:=\Hom RPR$ is perfect and the natural biduality map
\[
P\lra P^{**}=\Hom R{\Hom RPR}R
\]
is an isomorphism. Moreover for any $R$-complex $X$ the natural map
\[
P^{*}\otimes_{R} X \lra \Hom RPX
\]
is an isomorphism. In the sequel these properties are used without comment. 

\subsection{Levels}
\label{ssec:levels}
A \emph{length $l$ semiprojective filtration} of an $R$-complex $P$ is a sequence of 
$R$-subcomplexes of finitely generated projective modules
\[
0=\filt 0P \subseteq \filt 1P \subseteq \cdots \subseteq \filt lP=P
\]
such that $\nat{\filt{i-1}P}$ is a direct summand of $\nat{\filt{i}P}$ and the differential of $\filt iP/\filt{i-1}P$ is 
equal to zero, for $i=1,\dots,l$.  For every $R$-complex $F$, we set
  \[
\level {R}F = \inf\left\{l\in\bbN\  \left|\ \begin{gathered}F \text{ is a retract of some $R$-complex $P$ that} \\  \text{has a semiprojective filtration of length $l$} \end{gathered}\,\right\}\right.\,.
  \]
By \cite[2.4]{ABIM}, this number is equal to the \emph{level} $F$ with respect to $R$, defined in \cite[2.1]{ABIM}.

In particular, $\level RF$ is finite if and only if $F$ is quasi-isomorphic to some perfect complex. 
When $F$ is quasi-isomorphic to a perfect complex $P$, one has
\begin{equation}
\label{eq:level-span}
\level {R}F\le\cspan P\,.
\end{equation}
Indeed, if $P:=0\to P_{b}\to \cdots \to P_{a}\to 0$, then consider the filtration by subcomplexes $\filt nP:= P_{< n+a}$. 
The inequality can be strict; see \ref{ssec:koszul} below.

When $R$ is regular, any $R$-complex $F$ with $\hh F$ finitely generated satisfies 
\begin{equation}
\label{eq:level-regular}
\level {R}F\le \dim R+1\,.
\end{equation}

For $R$-complexes $X$ and $Y$  one has 
\begin{equation}
\label{eq:level-props}
\begin{aligned}
&\level R{(\shift^{i}X)}=\level RX \quad\text{for every integer $i$, and} \\
&\level R{(X\oplus Y)}=\max\{\level RX,\level RY\}.
\end{aligned}
\end{equation}
These equalities follow easily from the definitions.

\begin{lemma}
\label{lem:levels}
The following statements hold for every perfect $R$-complex $P$.
\begin{enumerate}[\quad\rm(1)]
\item
$\level {R}{(P^{*})} = \level {R}P$.
\item
For each perfect $R$-complex $Q$ there are inequalities
\begin{align*}
\level R{(P \otimes_{R} Q)} &\leq \level RP+\level RQ-1\,.\\
\level{R}{\Hom RPQ} &\leq \level RP+\level RQ-1\,.
\end{align*}
\end{enumerate}
\end{lemma}

\begin{proof}
(1) If $P$ is a retract of $P'$, then $\level RP\le \level R{P'}$ and $P^{*}$ is a retract of $(P')^{*}$. 
Thus, we can assume $P$ itself has a finite semiprojective filtration $\{\filt nP\}_{n=0}^{l}$. The 
inclusions $\filt {l-i}P\subseteq \filt {l-i+1}P \subseteq P$ define subcomplexes
\[
\filt i{P^*}:=\Ker(P^{*}\lra {\filt {l-i}P}^{*})\subseteq\Ker(P^{*}\lra{\filt {l-i+1}P})=:\filt {i+1}{P^*}
\]
of finitely generated projective modules.  They form a length $l$ semiprojective filtration of $P^*$, as 
$\nat{\filt{i-1}{P^*}}$ is a direct summand of $\nat{\filt{i}{P^*}}$ and there are isomorphisms
\[
\frac{\filt n{P^*}}{\filt {n-1}{P^*}}\cong \left(\frac{\filt {l-n+1}P}{\filt {l-n}P}\right)^{*}.
\]
This gives $\level R{P^{*}}\leq \level RP$, and the reverse inequality follows from $P\cong P^{**}$.

(2)  Assume first that $P$ has a semiprojective filtration $\{\filt n{P}\}_{n=0}^{l}$ and $Q$ has a 
semiprojective filtration $\{\filt n{Q}\}_{n=0}^{m}$.  For all $h,i$, we identity $\filt hP\otimes_{R}\filt iQ$ 
with a subcomplex of $P\otimes_{R}Q$. For each non-negative integer $n\ge 0$ form the subcomplex
\[
\filt nC: = \sum_{j\ges 0} \filt {j+1}P \otimes_{R}\filt {n-j}Q
\]
of $P\otimes_{R}Q$.  A direct computation yields an isomorphism of $R$-complexes
\[
\frac{\filt nC}{\filt {n-1}C} \cong \sum_{j\ges 0} \frac{\filt {j+1}P}{\filt jP}\otimes_{R} \frac{\filt {n-j}Q}{\filt {n-j-1}Q}\,.
\]
Thus $\{\filt nC\}_{n=0}^{l+m-1}$ is a semiprojective filtration of $P\otimes_{R}Q$.

The second inequality in (2) follows from the first one, given (1) and the isomorphism $\Hom RPQ\cong P^{*}\otimes_{R}Q$.
Next we verify the first inequality.  There is nothing to prove unless the levels of $P$ and $Q$ are finite.  
Thus we may assume that $P$ is a retract of a complex $P'$ with a semiprojective filtration of length  
$l=\level RP$ and $Q$ is a retract of a complex $Q'$ with a semiprojective filtration of length $m=\level RQ$.  
Then $P\otimes_{R}Q$ is a retract of $P'\otimes_{R}Q'$, and---by what we have just seen---this complex has a 
semiprojective filtration of length $l+m-1$, as desired.
  \end{proof}

\subsection{Ghost maps}
\label{ssec:ghosts}
A \emph{ghost} is a morphism $g\colon X\to Y$ in $\dcat R$ such that $\hh g=0$; see \cite[\S8]{Ch}. Evidently a composition of morphisms one of which is ghost is a ghost. 

The next result is a version of the ``Ghost Lemma''; cf.~\cite[Theorem~8.3]{Ch}, \cite[Lemma~4.11]{Roq}, and \cite[Proposition~2.9]{ABI}.

\begin{lemma}
\label{lem:ghost}
Let $F$ be an $R$-complex and $c$ an integer with $c\ge\level RF$.

When $g\colon X\to Y$ is a composition of $c$ ghosts the following morphisms are~ghosts
\begin{align*}
F\lotimes Rg &\colon F\lotimes RX\lra F\lotimes RY \quad\text{and}\\
\Rhom RFg &\colon \Rhom RFX\lra \Rhom RFY 
\end{align*}
  \end{lemma}

\begin{proof}
For every $R$-complex $W$ there is a canonical isomorphism
\[
\Rhom RFR\lotimes RW \xra{\ \simeq\ } \Rhom RFW\,,
\]
so it suffices to prove the first assertion.  For that, we may assume that $F$ has a semiprojective filtration $\{\filt nF\}_{n=0}^{l}$, where $l=\level RF$.  By hypothesis,  $g=h\circ f$ where $f\colon X\to W$  is a $(c-1)$fold composition of ghosts and $h\colon W\to Y$ is a ghost.  Tensoring these maps with the exact sequence of $R$-complexes
\[
0\lra \filt 1F \xra{\ \iota\ } F \xra{\ \pi\ }G \lra 0
\]
where $G:=F/\filt 1F$, yields a commutative diagram of graded $R$-modules
\[
\xymatrixcolsep{3pc}
\xymatrix{
\hh{ \filt 1F \otimes_{R} X} \ar@{->}[r] \ar@{->}[d]_{\hh{\filt 1F\otimes f}}
	& \hh{F \otimes_{R} X} \ar@{->}[r] \ar@{->}[d]^{\hh{F\otimes f}}
			& \hh{G \otimes_{R}X} \ar@{->}[d]^{\hh{G\otimes f}} \\
\hh{\filt 1F \otimes_{R} W} \ar@{->}[r]^{\hh{\iota\otimes W}} \ar@{->}[d]_{\hh{\filt 1F\otimes h}}
	& \hh{F \otimes_{R} W} \ar@{->}[r]^{\hh{\pi\otimes W}}  \ar@{->}[d]^{\hh{F\otimes h}}
			& \hh{G\otimes_{R}W} \ar@{->}[d]  \\
\hh{ \filt 1F \otimes_{R} Y} \ar@{->}[r]_{\hh{\iota\otimes Y}} & \hh{F \otimes_{R} Y} \ar@{->}[r] & \hh{G\otimes_{R}Y}
}
\]
where the rows are exact. Since $\level RG\le l-1\le c-1$, the induction hypothesis implies $G\otimes f$ is a ghost; 
that is to say, $\hh{G\otimes f}=0$. The commutativity of the diagram above and the exactness of the middle row implies that 
\[
\Image \hh{F\otimes f} \subseteq \Image\hh{\iota \otimes W}
\] 
This entails the inclusion below.
\begin{align*}
\Image \hh{F\otimes g} 
	& = \hh{F\otimes h}(\Image \hh{F\otimes f})\\
	& \subseteq \hh{F\otimes h}(\Image\hh{\iota \otimes W}) \\
	& \subseteq \Image \hh{\iota \otimes Y} \hh{\filt 1F\otimes h}) \\
	&=0
\end{align*}
The second equality comes  from the commutativity of the diagram. The last one holds because $\filt 1F$ 
is graded-projective and $\hh h=0$ imply $\hh{\filt 1F\otimes h}=0$. 
\end{proof}

\subsection{Koszul complexes}
\label{ssec:koszul}
Let $\bsx:=x_{1},\dots,x_{n}$ be elements in $R$. 

We write $K(\bsx)$ for the Koszul complex on $\bsx$. Thus $\nat {K(\bsx)}$ is the exterior algebra on 
a free $R$-module $K(\bsx)_1$ with basis $\{\wt x_1,\dots,\wt x_n\}$, and $\dd^K$ is the unique $R$-linear 
map that satisfies the Leibniz rule and has $\dd(\wt x_{i})=x_{i}$ for $i=1,\dots,n$.  In particular, $K(\bsx)$ is
a DG (differential graded) algebra, and so its homology $\hh{K(\bsx)}$ is a graded algebra with 
$\HH 0{K(\bsx)}=R/(\bsx)$.  This implies $(\bsx)\hh{K(\bsx)}=0$.

Evidently $K(\bsx)$ is a perfect $R$-complex; it is indecomposable when $R$ is local; see~\cite[4.7]{AGS}. As $K(\bsx)_i$ is 
non-zero precisely for $0,\dots, n$, from \eqref{eq:level-span} one gets
\[
\level R{K(\bsx)}\leq \cspan K(\bsx) = n+1\,.
\]
Equality holds if $R$ is local and $\bsx$ is a system of parameters; see Theorem~\ref{thm:init} below. 
However, $\cspan K(\bsx) - \level R{K(\bsx)}$ can be arbitrarily large; see \cite[Section~3]{AGS}.

For any Koszul complex $K$ on $n$ elements, there are isomorphisms of $R$-complexes
\[
K^{*}\cong \shift^{-n}K \quad\text{and}\quad K\otimes_{R}K \cong \bigoplus_{i=0}^{n}\shift^{i} K^{\binom ni}\,.
\]
See \cite[Propositions 1.6.10 and 1.6.21]{BH}.  It thus follows from \eqref{eq:level-props} that 
\begin{equation}
\label{eq:koszul-product}
\level R{\Hom RKK} = \level R{(K\otimes_{R}K)}=  \level RK\,.
\end{equation}
In particular, the inequalities in Lemma~\ref{lem:levels}(2) can be strict.

\section{Tensor nilpotent morphisms}
\label{sec:nilpotence}

In this section we prove the Tensor Nilpotence Theorem announced in the introduction.  We start by reviewing the properties of interest. 

\subsection{Tensor nilpotence}
\label{ssec:tensor-nilpotent}
Let $f\colon Y\to X$ be a morphism in $\dcat R$.

The morphism $f$ is said to be \emph{tensor nilpotent} if for some $n\in \bbN$ the morphism 
\[
\underbrace{f\lotimes R\cdots \lotimes Rf}_{n} \colon {Y\lotimes R\cdots \lotimes RY} \lra {X\lotimes R\cdots \lotimes RX}
\]
is equal to zero in $\dcat R$; when the $R$-complexes $X,Y$ are perfect this means that the morphism $\otimes^{n}f\colon \otimes^{n}_{R}Y\to\otimes^{n}_{R} X$ is homotopic to zero.  When $X$ is perfect and $f\colon X\to\shift^{l}X$ is a morphism with $\otimes^{n}f$ homotopic to zero the $n$-fold composition 
\[
X\lra \shift^{l}X \xra{\ \shift^{l}f\ } \shift^{2n}X \xra{\ \shift^{2l}f\ }\cdots \xra{\ \shift^{nl}\ } \shift^{nl}X
\]
is also homotopic to zero. The converse does not hold, even when $R$ is a field for in that case tensor nilpotent morphisms are zero.

\subsection{Fiberwise zero morphisms}
\label{ssec:fiberwise zero}
A morphism $f\colon Y\to X$ that satisfies 
\[
k(\fp)\lotimes Rf=0\quad\text{in}\quad \dcat{k(\fp)} \quad\text{for every $\fp\in\Spec R$}
\]
is said to be \emph{fiberwise zero}.  This is equivalent to requiring $k\lotimes Rf=0$ in $\dcat k$ for every homomorphism $R\to k$ with $k$ a field.  
In $\dcat k$ a morphism is zero if (and only if) it is a ghost, so the latter condition is equivalent to $\hh{k\lotimes Rf}=0$.

In $\dcat k$, a morphism is tensor nilpotent exactly when it is zero. Thus if $f$ is tensor nilpotent, it is fiberwise zero.  There is a partial converse:  \emph{If a morphism $f\colon G\to F$  of perfect $R$-complexes is  fiberwise zero, then it is tensor nilpotent.}  This was proved by Hopkins~\cite[Theorem~10]{Hop} and Neeman~\cite[Theorem~1.1]{Ne}.


The next result is the Tensor Nilpotence Theorem from the Introduction. Recall that an $R$-module is said to be $I$-torsion if each one of its elements is annihilated by some power of $I$. 

\begin{theorem}
\label{thm:fiberwisezero}
Let $R$ be a commutative noetherian ring and $f\colon G\to F$ a morphism of perfect $R$-complexes. 
If for some ideal $I$ of $R$ the following conditions hold
  \begin{enumerate}[\quad\rm(1)]  
\item
$f$ factors through some complex with $I$-torsion homology, and
\item
$\level{R}{\Hom RGF}\le \height I$\,,
\end{enumerate}
then $f$ is fiberwise zero. In particular, $f$ is tensor nilpotent.
  \end{theorem}

The proof of the theorem is given after Proposition \ref{prp:BCM}.

\begin{remark}
\label{rem:optimal}
Lemma~\ref{lem:levels} shows that the inequality (2) is implied by 
\[
 \level RF+\level RG\le \height I+1\,;
\]
the converse does not hold; see \eqref{eq:koszul-product}.  

On the other hand, condition (2) cannot be weakened: 
Let $(R,\fm,k)$ be a local ring and $G$ the  Koszul complex on some system of parameters of $R$ and  let  
\[
f\colon G\lra (G/G_{\les d-1}) \cong \shift^{d}R
\]
be the canonical surjection with $d=\dim R$. Then $G$ is an $\fm$-torsion complex and $\level {R}G=d+1$; 
see \ref{ssec:koszul}.  Evidently $\hh{k\otimes_{R}f}\ne 0$, so $f$ is not fiberwise zero. 
\end{remark}

In the proof of Theorem~\ref{thm:fiberwisezero} we exploit the functorial nature of $I$-torsion.

  \subsection{Torsion complexes}
     \label{ssec:torsion-complexes}
The derived $I$-torsion functor assigns to every $X$ in $\dcat R$ an $R$-complex $\lch_{I}X$; when $X$ is a module it computes its local cohomology: $\LC nIX=\HH {-n}{\lch_{I}X}$ holds for each integer $n$. There is a natural morphism $t\colon \lch_{I}X \lra X$ in $\dcat R$ that has the following universal property:  Every morphism  $Y\to X$ such that $\hh Y$ is $I$-torsion factors uniquely through $t$; see Lipman~\cite[Section~1]{Li}. It is easy to verify that the following conditions are equivalent.
\begin{enumerate}[\quad\rm(1)]
\item
$\hh X$ is $I$-torsion.
\item
${\hh X}_{\fp}=0$ for each prime ideal $\fp\not\supseteq I$.
\item
The natural morphism $t\colon \lch_{I}X\to X$ is a quasi-isomorphism.
\end{enumerate}

When they hold, we say that $X$ is \emph{$I$-torsion}.  Note a couple of properties:
\begin{align}
\label{eq:I-torsion-tensor}
&\text{If  $X$ is $I$-torsion, then $X\lotimes{R}Y$ is $I$-torsion  for any $R$-complex $Y$.}\\
\label{eq:I-torsion-monoidal}
&\text{There is a natural isomorphism } \lch_{I}(X\lotimes RY)\cong (\lch_{I}X)\lotimes RY\,.
 \end{align}

Indeed, $\hh {X_{\fp}}\cong {\hh X}_{\fp}=0$ holds for each $\fp\not\supseteq I$, giving $X_{\fp}=0$ in $\dcat R$. Thus
\[
(X\lotimes{R}Y)_{\fp}\cong X_{\fp}\lotimes {R} Y \cong 0
\] 
holds in $\dcat{R}$.  It yields $\hh{X\lotimes{R}Y}_{\fp}\cong \hh{(X\lotimes{R}Y)_{\fp}} =0$, as desired.

A proof of the  isomorphism in \eqref{eq:I-torsion-monoidal} can be found in \cite[3.3.1]{Li}.

\subsection{Big Cohen-Macaulay modules}
\label{ssec:BCM}
Let $(R,\fm,k)$ be a local ring.   

A (not necessarily finitely generated) $R$-module $C$ is \emph{big Cohen-Macaulay} if every system of  parameters for $R$ is a $C$-regular sequence, in the sense of  \cite[Definition 1.1.1]{BH}. In the literature the name is sometimes used for $R$-modules $C$ 
that satisfy the property for \emph{some} system of parameters for $R$; however, the $\fm$-adic completion of $C$   is then big Cohen-Macaulay in the sense above; see~\cite[Corollary~8.5.3]{BH}.

The existence of big Cohen-Macaulay was proved by Hochster~\cite{Hoc1,Hoc2} in case when $R$ contains a field as a subring, and by  Andr\'e~\cite{An} when it does not; for the latter case, see also Heitmann and Ma~\cite{HM}.

In this paper, big Cohen-Macaulay modules are visible only in the next result. 

 \begin{proposition} 
  \label{prp:BCM}
Let $I$ be an ideal in $R$ and set $c:=\height I$. 

When $C$ is a big Cohen-Macaulay $R$-module the following assertions hold.
\begin{enumerate}[\quad\rm(1)]
\item
The canonical morphism $t\colon \lch_{I}C \to C$ from the $I$-torsion complex $\lch_{I}C$ 
(see \emph{\ref{ssec:torsion-complexes}}) is a composition of $c$ ghosts. 
\item
If a morphism $g\colon G\to C$ of $R$-complexes with $\level RG\le c$ factors through some $I$-torsion complex, then $g=0$.
\end{enumerate}
  \end{proposition}
   
   \begin{proof} 
(1)  We may assume $I=(\bsx)$, where $\bsx=\{x_{1},\dots,x_{c}\}$ is part of a system of parameters for $R$; 
see \cite[Theorem~A.2]{BH}.  The morphism $t$ factors as
   \[
\lch_{(x_{1},\dots,x_{c})}(C) \lra \lch_{(x_{1},\dots,x_{c-1})}(C) \lra \cdots \lra \lch_{(x_{1})}(C) \lra C\,.
   \]
Since the sequence $x_{1},\dots,x_{c}$ is $C$-regular, we have $\HH i{\lch_{(x_{1},\dots,x_{j})}(C)}=0$ for $i\ne -j$; see 
\cite[(3.5.6) and (1.6.16)]{BH}.  Thus every one of the arrows above is a ghost, so that $t$ is a composition of $c$ ghosts, as desired.

(2)  Suppose $g$ factors as $G\to X\to C$ with $X$ an $I$-torsion $R$-complex. As noted in \ref{ssec:torsion-complexes}, the morphism $X\to C$ factors through $t$, so $g$ factors as
\[
G\xra{g'} X \xra{g''} \lch_{I}C \xra{t} C\,.
\]
In view of the hypothesis $\level {R}G\le c$ and part (1), Lemma~\ref{lem:ghost} shows that   
  \[
\Rhom RG{t} \colon \Rhom RG{\lch_{I}C}\lra \Rhom RGC
  \]
is a ghost.  Using brackets to denote cohomology classes, we get 
  \[
[g]=[tg''g']=\HH 0{\Rhom RG{t}}([g''g'])=0\,.
  \]  

Due to the isomorphism \eqref{eq:complexes1}, this means that $g$ is zero in $\dcat R$.
   \end{proof}

\begin{lemma} 
  \label{lem:fiberwise}
Let $f\colon G\to F$ be a morphism of perfect $R$-complexes, where $G$ is finite free with $G_{i}=0$ for $i\ll 0$ 
and $F$ is perfect. Let $f'\colon F^{*}\otimes_{R} G\to R$ denote the composed morphism in the next display,
where $e$ is the evaluation map:
  \begin{equation*}
F^{*}\otimes_{R} G\xra{\ F^{*}\otimes_{R}f\ } F^{*}\otimes_{R} F\xra{\ e\ } R\,.
  \end{equation*}

If $f$ factors through some $I$-torsion complex, then so does $f'$.

The morphism $f'$ is fiberwise zero if and only if so is $f$.
  \end{lemma}

  \begin{proof} 
For the first assertion, note that if $f$ factors through an $I$-torsion complex $X$, then $F^{*}\otimes_{R}f$  factors through $F^{*}\otimes_{R}X$, and the latter is $I$-torsion.

For the second assertion, let $k$ be field and $R\to k$ be a homomorphism of rings. Let $\ov{(-)}$ and $(-)^{\vee}$ stand for 
$k\otimes_{R}(-)$ and $\Hom k-k$, respectively.  The goal is to prove that $\ov f=0$ is equivalent to $\ov{f'}=0$.

Since $F$ is perfect, there are canonical isomorphisms
\[
F^{*}\otimes k \xra{\ \cong\ } \Hom RFk \cong \Hom k{k\otimes_{R}F}k={(\ov F)}^{\vee}\,.
\]
Given this, it follows that $\ov{f'}$ can be realized as the composition of morphisms
\[
{(\ov F)}^{\vee} \otimes_{k}\ov G  \xra{\ {(\ov F)}^{\vee} \otimes_{k}\ov f   } {(\ov F)}^{\vee} \otimes_{k} \ov F \xra{\ \ov{e}\ } k \,.
\]
If $\ov F$ is zero, then $\ov f=0$ and $\ov{f'}=0$ hold. When $\ov F$ is nonzero,  it is easy to verify that $\ov f\ne 0$ is equivalent to $\ov{f'}\ne 0$, as desired.
  \end{proof} 

\begin{proof}[Proof of Theorem~\emph{\ref{thm:fiberwisezero}}]
Given morphisms of $R$-complexes $G\to X\to F$ such that $F$ and $G$ are perfect and $X$ is $I$-torsion for an ideal $I$ with 
\[
 \level{R}{\Hom RGF}\le \height I \,,
\]
we need to prove that $f$ is fiberwise zero. This implies the tensor nilpotence of $f$, as recalled in \ref{ssec:tensor-nilpotent}.

By Lemma~\ref{lem:fiberwise}, the morphism $f'\colon F^{*}\otimes_{R} G\to R$  factors through an $I$-torsion complex, and if $f'$ is fiberwise zero, so if $f$. The isomorphisms of $R$-complexes
\[
(F^{*}\otimes_{R}G)^{*} \cong G^{*} \otimes_{R} F \cong \Hom RGF
\]
and Lemma~\ref{lem:levels} yield $\level R{(F^{*}\otimes_{R}G)}=\level R{\Hom RGF}$. Thus, replacing $f$ by $f'$, it suffices to prove that if $f\colon G\to R$ is a morphism that factors through an $I$-torsion complex and satisfies $\level RG\le \height I$, then $f$ is fiberwise zero.

Fix $\fp$ in $\Spec R$. When $\fp\not\supseteq I$ we have $X_{\fp}=0$, by \ref{ssec:torsion-complexes}(2). For $\fp\supseteq I$ we have 
\[
\level{R_{\fp}}{G_{\fp}}\le  \level RG \le \height I \le \height I_{\fp}\,,
\]
where the first inequality follows directly from the definitions; see \cite[Proposition 3.7]{ABIM}.   It is easy to verify that $X_{\fp}$ is $I_{\fp}$-torsion.  Thus,  localizing at $\fp$, we may further  assume  $(R,\fm,k)$ is a local ring, and we have to prove that $\hh{k\otimes_{R}f}=0$ holds.  

Let $C$ be a big Cohen-Macaulay $R$-module.  It satisfies $\fm C\ne C$, so the canonical map $\pi\colon R\to k$ factors as $R\xra{\gamma}C\xra{\varepsilon}k$.  The composition  $G\xra{f}R\xra{\gamma}C$ is zero in $\dcat R$, by Proposition~\ref{prp:BCM}.  We get $\pi f=\varepsilon\gamma f=0$, whence $\hh{k\otimes_{R}\pi}\hh{k\otimes_{R}f}=0$. Since $\hh{k\otimes_{R}\pi}$ is bijective, this implies $\hh{k\otimes_{R}f}=0$, as desired.
  \end{proof}

The following consequence of Theorem~\ref{thm:fiberwisezero} is often helpful.
 
\begin{corollary}
\label{cor:overlap}
Let $(R,\fm,k)$ be a local ring, $F$ a perfect $R$-complex, and $G$ an $R$-complex 
of finitely  generated free modules.

If a morphism of $R$-complexes $f\colon G\to F$ satisfies the conditions 
\begin{enumerate}[\quad\rm(1)]
\item
$f$ factors through some $\fm $-torsion complex, 
\item
$\sup \nat F - \inf \nat G \le \dim R-1$, and
\end{enumerate}
then $\hh{k \otimes_{R}f}=0$.
\end{corollary}

\begin{proof}
An $\fm$-torsion complex $X$ satisfies $k(\fp)\lotimes RX=0$ for any $\fp$ in $\Spec R\setminus \{\fm\}$. Thus a morphism, $g$, of $R$-complexes that factors through $X$ is fiberwise zero if and only if $k\lotimes Rg=0$. This remark will be used in what follows. 

Condition (2) implies $G_{n}=0$ for $n\ll 0$. Let $f'$ denote the composition
\[
F^{*}\otimes_{R}G \xra{\ F^{*}\otimes_{R}f\ } F^{*}\otimes_{R}F \xra{\ e\ } R\,,
\]
where $e$ is the evaluation map.  Since $\inf \nat{(F^{*}\otimes_{R} G)} = - \sup \nat F + \inf \nat G$, Lemma~\ref{lem:fiberwise} shows that it suffices to prove the corollary 
for morphisms $f\colon G\to R$.  

As $f$ factors through some $\fm$-torsion complex, so does the composite morphism
\[
G_{\les 0} \subseteq G \xra{\ f\ }R
\]
It is easy to check that if the induced map 
$\hh{k\otimes_{R}G_{\les 0}} \to \hh{k\otimes_{R}R} = k$ is zero, then so is $\hh{k\otimes_{R}f}$.  
Thus we may assume $G_{n} =0$ for $n\not\in [-d+1,0]$, where $d=\dim R$. This implies $\level {R} G \le d$, 
so Theorem~\ref{thm:fiberwisezero} yields the desideratum.
\end{proof}

For some applications the next statement, with weaker hypothesis but also weaker conclusion, suffices.  The example in Remark~\ref{rem:optimal} shows that the result cannot be strengthened to conclude that $f$ is fiberwise zero.

\begin{theorem}
\label{thm:ghost}
Let $R$ be a local ring and $f\colon G\to F$ a morphism of $R$-complexes. 

If there exists an ideal $I$ of $R$ such that 
\begin{enumerate}[\quad\rm(1)]
\item
$f$ factors through an $I$-torsion complex, and
\item
$\level RF\le \height I$,
\end{enumerate}
then $\hh{C\lotimes {R} f}=0$ for every big Cohen-Macaulay module $C$.
\end{theorem}

\begin{proof}
Set $c:=\height I$ and let $t\colon \lch_{I}C\to C$ be the canonical morphism. It follows from \eqref{eq:I-torsion-tensor} that $C\lotimes Rf$ also factors through an $I$-torsion $R$-complex. Then the quasi-isomorphism \eqref{eq:I-torsion-monoidal} and the universal property of the derived $I$-torsion functor, see \ref{ssec:torsion-complexes}, implies that $C\lotimes Rf$ factors as a composition of the morphisms:
\[
C\lotimes RG \lra (\lch_{I}C) \lotimes R F \xra{\ t\lotimes RF\ }  C\lotimes R F
\]
By Proposition~\ref{prp:BCM}(1) the morphism $t$  is a composition of $c$ ghosts.   Thus condition (2) and Lemma~\ref{lem:ghost} imply $t\lotimes RF$ is a ghost, and hence so is $C\lotimes Rf$.
\end{proof}

\section{Applications to local algebra}
\label{sec:applications}
In this section we record applications the Tensor Nilpotence Theorem to local algebra. To that end it is expedient 
to reformulate it as the Morphic Intersection Theorem from the Introduction, restated below.

\begin{theorem}
\label{thm:MIT}
Let $R$ be a commutative noetherian ring and $f\colon G\to F$ a morphism of perfect $R$-complexes.

If $f$ is not fiberwise zero and factors through a complex with $I$-torsion homology for an ideal $I$ of $R$, then
there are inequalities:
\[
\cspan F + \cspan G -1 \ge  \level{R}{\Hom RGF} \ge  \height I + 1\,. 
\]
  \end{theorem}

\begin{proof}
The inequality on the left comes from Lemma~\ref{lem:levels} and \eqref{eq:level-span}. The one on the right is  
the contrapositive of Theorem~\ref{thm:fiberwisezero}.
\end{proof}

Here is one consequence.

\begin{theorem}
\label{thm:init}
Let $R$ be a local ring and $F$ a complex of finite free $R$-modules:
\[
F:=\quad 0\to F_{d}\to F_{d-1}\to \cdots \to F_{0}\to 0
\]
For each ideal $I$ such that $I\cdot \HH iF=0$ for $i\ge 1$ and $I\cdot z=0$ for some element $z$ in $\HH 0F \setminus \fm \HH 0F$, where $\fm$ is the maximal ideal of $R$, one has 
\[
d +1 \ge \cspan F \ge \level RF\ge \height I +1\,.
\]
\end{theorem}

\begin{proof}
Indeed, the first two inequalities are clear from definitions. As to third one, pick $\wt z\in F_{0}$ representing $z$ in  $\HH  0F$ and consider the morphism of complexes $f\colon R\to F$ given by $r\mapsto r\wt z$.  Since $z$ is not in $\fm \HH 0F$, one has 
\[
\HH 0{k\otimes_{R}f} = k\otimes_{R}\HH 0f \ne 0
\]
for $k=R/\fm$.  In particular, $k\otimes_{R}f$ is nonzero. On the other hand, $f$ factors through the 
inclusion $X\subseteq F$, where $X$ is the subcomplex defined by 
\[
X_{i} = \begin{cases}
F_{i} & i\ge 1 \\
R\wt z + d(F_{1}) & i=0
\end{cases}
\]
By construction, we have $\HH iX=\HH iF$ for $i\ge 1$ and $I\HH 0X=0$, so $\hh X$ is $I$-torsion. 
The desired inequality now follows from Theorem~\ref{thm:MIT} applied to $f$.
\end{proof}

The preceding result is a stronger form of the Improved New Intersection Theorem\footnote{This and the other statements in this section were conjectures prior to the appearance of \cite{An}.} of Evans and Griffith~\cite{EG}; see also~\cite[\S2]{Hoc4}. First,  the latter is in terms of spans of perfect complexes whereas the one above is in terms of levels with respect to $R$; second,  the hypothesis on the homology of $F$ is weaker. Theorem~\ref{thm:init}  also subsumes prior extensions of the New Intersection Theorem to statements involving levels, namely ~\cite[Theorem~5.1]{ABIM}, where it was assumed that $I\cdot \HH 0F=0$ holds, and \cite[Theorem~3.1]{AGS} which requires $\HH iF$ to have finite length for $i\ge 1$.

In the influential paper \cite{Hoc3}, Hochster identified certain \emph{canonical elements} in the local cohomology of local rings, conjectured that they are never zero, and proved that statement in the equal characteristic case.   He also gave several reformulations that do not involve local cohomology. Detailed discussions of the relations between these statements and the histories of their proofs are presented in \cite{Ro1} and \cite{Hoc5}. 

Some of those statements concern properties of morphisms from the Koszul complex on some system of parameters to resolutions of various $R$-modules.  This makes them particularly amenable to approaches from the Morphic Intersection Theorem.  In the rest of this section we uncover direct paths to various forms of the Canonical Element Theorem and related results.

We first prove a version of \cite[2.3]{Hoc3}.  The conclusion there is that $f_d$ is not zero, but the remarks in 
\cite[2.2(6)]{Hoc3} show that it is equivalent to the following statement. 

  \begin{theorem}
\label{thm:cec-plus}
Let $(R,\fm,k)$ be a local ring, $\bsx$ a system of parameters for $R$, and $K$ the Koszul complex on $\bsx$, 
and $F$ a complex of free $R$-modules.

If $f\colon K \to F$ is a morphism of $R$-complexes with $\HH 0{k\otimes_{R}f}\ne 0$, then one has
  \[
\HH d{S\otimes_{R}f}\ne0 \quad\text{for}\quad S=R/(\bsx) \quad\text{and}\quad d=\dim R\,.
  \]
  \end{theorem}  

\begin{proof}
Recall from \ref{ssec:koszul} that $\dd(K)$ lies in $(\bsx)K$, that $K_1$ has a basis $\wt x_1,\dots,\wt x_d$ and 
$\nat K$ is the exterior algebra on $K_1$.  Thus $K_d$ is a free $R$-module with basis $x=\wt x_1\cdots \wt x_d$ 
and $\HH d{S\otimes_RK}=S(1\otimes x)$, so we need to prove $f(K_{d})\not\subseteq (\bsx)F_{d}+\dd(F_{d+1})$.  

Arguing by contradiction, we suppose the contrary.  This means 
  \[
f(x) = x_1 y_1 +\cdots + x_dy_d +\dd^F(y)
  \]
with $y_1,\dots,y_d\in F_d$ and $y\in F_{d+1}$.  For $i=1,\dots,d$ set $x^{*}_i:=\wt x_1\cdots\wt x_{i-1}\wt x_{i+1}\cdots\wt x_d$; 
thus $\{ x^{*}_1,\dots,x^{*}_d\}$ is a basis of the $R$-module $K_{d-1}$.  Define $R$-linear maps 
  \begin{alignat*}{3}
h_{d-1}\colon K_{d-1}&\to F_{d} &\quad&\text{by}\quad &h_{d-1}(x^{*}_i)&=(-1)^{i-1}y_i \quad\text{for}\quad i=1,\dots,d\,.
  \\
h_{d}\colon K_{d}&\to F_{d+1} &\quad&\text{by}\quad &h_{d}(x)&=y\,.
  \end{alignat*}
Extend them to a degree one map $h\colon K\to F$ with $h_{i}=0$ for $i\ne d-1,d$. The map
\[
g:=f - \partial^Fh-h\partial^K\colon K\to F
\]
is easily seen to be a morphism of complexes that is homotopic to $f$ and satisfies $g_{d}=0$. This last 
condition implies that $g$ factors as a composition of morphisms
\[
K \xra{\ g'\ } F_{<d}\subseteq F\,.
\]
The complex $K$ is $\fm$-torsion; see \ref{ssec:koszul}.  Thus Corollary~\ref{cor:overlap}, applied to $g'$, yields 
$\hh{k\otimes_{R}g'}=0$. This gives the second equality below:
\[
\HH 0{k\otimes_{R}f}=\HH 0{k\otimes_{R}g} = \HH 0{k\otimes_{R}g'}=0\,.
\]
The first one holds because $f$ and $g$ are homotopic, and the last one because $g_{0}=g'_{0}$. 

The result of the last computation contradicts the hypotheses on $\HH 0{k\otimes_{R}f}$.
\end{proof}

A first specialization is the Canonical Element Theorem.

\begin{corollary}
\label{cor:cec}
Let $I$ be an ideal in $R$ containing a system of parameters $x_{1},\dots,x_{d}$. With $K$ the Koszul complex on $\bsx$ and $F$ a free resolution of $R/I$,  any morphism $f\colon K\to F$ of $R$-complexes lifting the surjection $R/(\bsx)\to R/I$ has  $f_{d}(K) \ne 0$.\qed
\end{corollary}

As usual, when $A$ is a matrix, $I_{d}(A)$ denotes the ideal of its minors of size $d$.

\begin{corollary}
\label{cor:sop}
Let $R$ be a local ring, $\bsx$ a system of parameters for $R$, and $\bsy$ a finite subset of $R$ with $(\bsy)\supseteq (\bsx)$.

If $A$ is a matrix such that $A\bsy = \bsx$, then $I_{d}(A)\not\subseteq (\bsx)$ for $d=\dim R$.
\end{corollary}

\begin{proof}
Let $K$ and $F$ be the Koszul complexes on $\bsx$ and $\bsy$, respectively. The matrix $A$ defines a unique morphism of DG $R$-algebras $f\colon K \to F$. Evidently, $\HH 0{k\otimes_{R}f}$ is the identity map on $k$, and hence is not zero.  Since $f_{d}$ can be represented by a column matrix whose entries are the various $d\times d$ minors of $A$, the desired statement is a direct consequence of Theorem~\ref{thm:cec-plus}.
\end{proof}

A special case of the preceding result yields the Monomial Theorem.

\begin{corollary}
\label{cor:mc}
When $y_{1},\dots,y_{d}$ is a system of parameters for local ring, one has
\[
(y_{1}\cdots y_{d})^{n}\not\in (y_{1}^{n+1},\dots,y_{d}^{n+1})
  \quad\text{for every integer}\quad n\ge1 \,.
\]
\end{corollary}

\begin{proof}
\pushQED{\qedhere}
Apply Corollary~\ref{cor:sop} to the inclusion $(y_{1}^{n+1},\dots,y_{d}^{n+1})\subseteq (y_{1},\dots,y_{d})$ and
\begin{xxalignat}{3}
&{\phantom{\square}}
&A:&=
\begin{bmatrix}
y_{1}^{n} & 0              & \cdots & 0 \\
0             & y_{2}^{n}  & \cdots & 0 \\
\vdots     &\vdots        &\ddots  & \vdots \\
0             & 0              & \cdots & y_{d}^{n}
\end{bmatrix} \,.
&&\square
\end{xxalignat}
\end{proof}
 
We also deduce from Theorem~\ref{thm:cec-plus} another form of the Canonical Element Theorem. Roberts \cite{Ro} proposed the statement and proved that it is equivalent to the Canonical Element Theorem; a different proof appears in Huneke and Koh~\cite{HK}.   

Recall that for any pair $(S,T$) of $R$-algebras the graded module $\Tor {}RST$ carries a natural structure of graded-commutative $R$-algebra, given by the $\pitchfork$-product of Cartan and Eilenberg~ \cite[Chapter XI.4]{CE}. 

\begin{lemma}
\label{lem:tor}
Let $R$ be a commutative ring, $I$ an ideal of $R$, and set $S:=R/I$.  Let $G\to S$ be some $R$-free resolution,
$K$ be the Koszul complex on some generating set of~$I$, and $g\colon K\to G$ a morphisms of $R$-complexes 
lifting the identity of $S$.

For every surjective homomorphism $\psi\colon S\to T$ of of commutative rings  there is a commutative diagram of 
strictly graded-commutative $S$-algebras
\[
\xymatrixcolsep{.75pc}
\xymatrixrowsep{2.5pc}
\xymatrix{
S\otimes_{R}K \ar@{=}[r] 
	& \bigwedge_{S}\HH 1{S\otimes_{R}K} \ar@{->>}[rr]^-{\alpha} 
			&&\bigwedge_{S} \Tor 1RSS \ar@{->}[rr]^-{\mu^{S}} \ar@{->>}[d]_-{\bigwedge_{\psi}\Tor {1}RS{\psi}}
				&& \Tor {}RSS \ar@{->}[d]^-{\Tor {}RS{\psi}} \\
	& 		&&\bigwedge_{T} \Tor 1RST \ar@{->}[rr]^-{\mu^{T}}   
				&&  \Tor {}RST
}
\]
where $\alpha_1=\HH 1{S\otimes_Rg}$, the map $\alpha$ is defined by the functoriality of exterior algebras, and 
the maps $\mu^?$ are defined by the universal property of exterior algebras.
  \end{lemma}

  \begin{proof}
The equality follows from $\dd^{K}(K)\subseteq IK$ and $\nat K =\bigwedge_{R}K_1$.  The resolution $G$ can be chosen to have 
$G_{\les1}=K_{\les1}$; this makes $\alpha_1$ surjective, and the surjectivity of $\alpha$ follows.  The map $\Tor {1}RS{\psi}$ 
is surjective because it can be identified with the natural map $I/I^2\to I/IJ$, where $J=\Ker(R\to T)$; the surjectivity of 
$\bigwedge_{\psi}\Tor {1}RS{\psi}$ follows.  The square  commutes by the naturality of $\pitchfork$-products.
  \end{proof}

\begin{theorem}
  \label{thm:wedge}
Let $(R,\fm,k)$ be a local ring, $I$ a parameter ideal, and $S:=R/I$. For each surjective homomorphism 
$S\to T$ the morphism of graded $T$-algebras
\[
\mu^{T}\colon \textstyle{\bigwedge}_{T} \Tor 1RST \lra \Tor {}RST
\]
has the property that $\mu^{T}\otimes_{T}k$ is injective. In particular, $\mu^{k}$ is injective.
  \end{theorem}

\begin{proof}
The functoriality of the construction of $\mu$ implies that the canonical surjection $\pi\colon T\to k$ 
induces a commutative diagram of graded-commutative algebras
\[
\xymatrixcolsep{1.5pc}
\xymatrixrowsep{2.5pc}
\xymatrix{
\bigwedge_{T} \Tor 1RST \ar@{->}[r]^-{\mu^{T}} \ar@{->}[d]_-{\bigwedge_{\pi}\Tor {1}RS{\pi}}
	&  \Tor {}RST \ar@{->}[d]^{\Tor{}RS\pi} \\
\bigwedge_{k} \Tor 1RSk \ar@{->}[r]^-{\mu^{k}} &  \Tor {}RSk 
}
\]
It is easy to verify that $\pi$ induces a bijective map 
  \[
\Tor 1RS{\pi}\otimes_{T}k\colon \Tor 1RST\otimes_{T}k\to \Tor 1RSk\,,
  \]
so $(\wedge_{\pi}\Tor {}RS{\pi})\otimes_{T}k$ is an isomorphism.  Thus it suffices to show $\mu^{k}$ is injective.

Let $K$ be the Koszul complex on a minimal generating set of  $I$.  Let $G\xra{\simeq}S$ and $F\xra{\simeq} k$ be 
$R$-free resolutions of $S$ and $k$, respectively.  Lift the identity map of $S$ and the canonical surjection $\psi\colon S\to k$ to morphisms $g\colon K\to G$  and $h\colon G\to F$, respectively. We have $\mu^{S}\alpha=\hh{S\otimes_Rg}$ and $\Tor {}RS{\psi}=\hh{S\otimes_R h}$.  This implies the second equality in the string
  \[
\mu^k_d\Tor {d}RS{\pi}\alpha_d = \Tor {d}RS{\psi}\mu^{S}_d\alpha_d=\HH d{S\otimes_{R}hg}\ne0 \,.
  \]
The first equality comes from Lemma \ref{lem:tor}, with $T=k$, and the non-equality from Theorem~\ref{thm:cec-plus}, 
with $f=hg$.  In particular, we get $\mu^k_d\ne0$.  We have an isomorphism $\Tor {1}RSk\cong\bigwedge_k k^d$ 
of graded
$k$-algebras, so $\mu^{k}$ is injective by the next remark.
  \end{proof}

  \begin{remark}
  \label{rem:wedge}
If $Q$ is a field, $d$ is a non-negative integer, and $\lambda\colon\bigwedge_Q Q^d \to B$ is a homomorphism of graded 
$Q$-algebras with $\lambda_d\ne0$, then $\lambda$ is injective.

Indeed, the graded subspace $\bigwedge^d_Q Q^d$ of exterior algebra $\bigwedge_Q Q^d$ is contained in every every 
non-zero ideal and has rank one, so $\lambda_d\ne0$ implies $\Ker(\lambda)=0$.
  \end{remark}

\section{Ranks in finite free complexes}
\label{sec:rank}

This section is concerned with DG modules over Koszul complexes on sequences of parameters.
Under the additional assumptions that $R$ is a domain and $F$ is a resolution of some $R$-module, the theorem 
below was proved in \cite[6.4.1]{Av}, and earlier for cyclic modules in \cite[1.4]{BE}; background is reviewed after 
the proof.  The Canonical Element Theorem, in the form of Theorem~\ref{thm:cec-plus} above, is used in the proof.
 
\begin{theorem}
\label{thm:rank}
Let $(R,\fm,k)$ be a local ring, set $d=\dim R$, and let $F$ be a complex of finite free $R$-modules with $\HH 0F\ne 0$ and $F_{i}=0$ for $i<0$.

If $F$ admits a structure of DG module over the Koszul complex on some system of parameters of $R$, then there is an inequality
  \begin{equation}
    \label{eq:rank}
\rank_{R} (F_{n}) \ge \binom dn \quad\text{for each}\quad n\in\bbZ\,.
  \end{equation}
  \end{theorem}

\begin{proof}
The desired inequality is vacuous when $d=0$, so suppose $d\ge 1$.  Let $\bsx$ be the said system of parameters of $R$ and $K$ the Koszul complex on $\bsx$.  Since $F$ is a DG $K$-module, each $\HH iF$ is an $R/(\bsx)$-module, and hence of finite length.

First we reduce to the case when $R$ is a domain. To that end, let $\fp$ be a prime ideal of $R$ such that $\dim (R/\fp)=d$. Evidently, the image of $\bsx$ in $R$ is a system of parameters for $R/\fp$. By base change, $(R/\fp) \otimes_{R}F$ is a DG module over $(R/\fp)\otimes_{R}K$, the Koszul complex on $\bsx$ with coefficients in $R/\fp$, with
\[
\HH 0{(R/\fp) \otimes_{R}F}\cong R/\fp\otimes \HH 0F\ne 0\,.
\]
Moreover, the rank of $\nat F$ as an $R$-module equals the rank of $\nat{(R/\fp)\otimes_{R}F}$ as an $R/\fp$-module. Thus, after base change to $R/\fp$ we can assume $R$ is a domain.

Choose a cycle $z\in F_{0}$ that maps to a minimal generator of the $R$-module $\HH 0F$. Since $F$ is a DG $K$-module, this yields a morphism of DG $K$-modules 
\[
f\colon K\to F \quad\text{with}\quad f(a)=a z\,.
\]
This is, in particular, a morphism of complexes. Since $k\otimes_{R}\HH 0F\ne 0$, by the choice of $z$, Theorem~\ref{thm:cec-plus} applies, and yields that $f(K_{d})\ne 0$. Since $R$ is a domain, this implies  $f(Q\otimes_{R}K_{d})$ is non-zero, where $Q$ is the field of fractions of $R$. 

Set $\Lambda:=\nat{(Q\otimes_{R}K)}$ and consider the homomorphism of graded $\Lambda$-modules
\[
\lambda:= Q\otimes_{R}\nat f\colon \Lambda\to Q\otimes_{R}\nat F\,.
\]
As $\Lambda$ is isomorphic to $\bigwedge_Q Q^d$, Remark  \ref{rem:wedge} gives the inequality in the display
\[
\rank_{R}(F_{n}) =\rank_{Q}(Q\otimes_{R}F_{n}) \geq \rank_{Q}(\Lambda_{n}) = \binom dn\,.
\]
Both equalities are clear.
\end{proof}

The inequalities \eqref{eq:rank} are related to a major topic of research in commutative algebra.  We discuss
it for a local ring $(R,\fm,k)$ and a bounded $R$-complex $F$ of finite free modules with 
$F_{<0}=0$, homology of finite length, and $\HH 0F\ne 0$.

\subsection{Ranks of syzygies}
  \label{ssec:total}
The celebrated and still open Rank Conjecture of Buchsbaum and Eisenbud \cite[Proposition~1.4]{BE}, and Horrocks \cite[Problem 24]{Ha}  predicts that \eqref{eq:rank} holds whenever $F$ is a \emph{resolution} of some module of finite length.  

That conjecture is known for $d\le4$.  Its validity would imply $\sum_{n} \rank_R F_n\ge2^d$.  For $d=5$ and equicharacteristic 
$R$, this was proved in \cite[Proposition~1]{AB} by using Evans and Griffth's Syzygy Theorem \cite{EG}; in view of~\cite{An}, 
it holds for all $R$.
 
In a breakthrough, M.~Walker  \cite{Wa} used methods from K-theory to prove that $\sum_{n} \rank_R F_n\ge2^d$ holds 
when $R$ contains $\frac12$ and is complete intersection (in particular, regular), and when $R$ is an algebra over some 
field of positive characteristic.

\subsection{Obstructions for DG module structures}
  \label{ssec:obs}
Theorem \ref{thm:rank} provides a series of obstruction for the existence of any DG module structure on $F$.  In particular, it 
implies that if $\rank_R F< 2^{d}$ holds with $d=\dim R$, then $F$ supports no DG module structure over $K(\bsx)$ 
for \emph{any} system of parameters $\bsx$.  Complexes satisfying the restriction on ranks were recently constructed 
in \cite[4.1]{IW}. These complexes have nonzero homology in degrees $0$ and $1$, so they are not resolutions of $R$-modules.

\subsection{DG module structures on resolutions}
  \label{ssec:obs}
Let $F$ be a minimal resolution of an $R$-module $M$ of nonzero finite length and $\bsx$ a parameter set for~$R$ with $\bsx M=0$.  

When $F$ admits a DG module structure over $K(\bsx)$ the Rank Conjecture holds, by Theorem \ref{thm:rank}.  It was conjectured in \cite[1.2$'$]{BE} that such a structure exists for all $F$ and~$\bsx$.  An obstruction to its existence was found in \cite[1.2]{Av}, and examples when that obstruction is not zero were produced in \cite[2.4.2]{Av}.  On the other hand, by \cite[1.8]{Av} the obstruction vanishes when $\bsx$ lies in $\fm\ann_R(M)$.  

It is not known if $F$ supports some DG $K(\bsx)$-module structure for special choices of~$\bsx$; in particular, 
for high powers of systems of parameters contained in $\ann_R(M)$.

 
\end{document}